\documentclass[reqno,12pt,letterpaper]{amsart}
\usepackage{amsmath,amssymb,amsthm,graphicx,mathrsfs,url}
\usepackage[usenames,dvipsnames]{color}
\usepackage[colorlinks=true,linkcolor=Red,citecolor=Green]{hyperref}

\def\?[#1]{\textbf{[#1]}\marginpar{\Large{\textbf{??}}}}
\newtheorem{prop}{Proposition}
\newtheorem{thm}[prop]{Theorem}

\newtheorem{lem}[prop]{Lemma}

\newtheorem{rem}[prop]{Remark}

\numberwithin{equation}{section}
\numberwithin{prop}{section}

\DeclareMathOperator{\Op}{Op}

\DeclareMathOperator{\supp}{supp}

\begin{document}
\title{Control for Schr\"{o}dinger equation on hyperbolic surfaces}

\author{Long Jin}
\email{long249@purdue.edu}
\address{Department of Mathematics, Purdue University,
150 N. University St, West Lafayette, IL 47907}

\begin{abstract} 
We show that any nonempty open set on a hyperbolic surface provides
observability and control for the time dependent Schr\"odinger equation.
The only other manifolds for which this was previously known are flat tori  \cite{jaffard, haraux, komornik}. The proof is based on the main estimate in \cite{messupp} and standard arguments in control theory.
\end{abstract}

\maketitle

\section{Introduction}
Let $M$ be a compact (connected) hyperbolic surface and $\Delta$ the Laplace-Beltrami operator on $M$. In a recent paper with Dyatlov, \cite{messupp}, we prove the following semiclassical control result which roughly says that any open set in $S^\ast M$ controls the whole $S^\ast M$ in the $L^2$-sense.
\begin{thm}{\cite[Theorem 2]{messupp}}
\label{t:semi-control}
Assume that $a\in C_0^\infty(T^\ast M)$ and $a|_{S^\ast M}\not\equiv0$, then there exist constants $C$, $h_0>0$ only depending on $M$ and $a$ such that for all $0<h<h_0$ and $u\in H^2(M)$,
\begin{equation}
\label{e:semi-control}
\|u\|_{L^2(M)}\leq C\|\Op_h(a)u\|_{L^2(M)}+C\frac{\log(1/h)}{h}\|(-h^2\Delta-1)u\|_{L^2(M)}.
\end{equation}
\end{thm}

In this short notes, we show that Theorem \ref{t:semi-control} implies the following observability result of the Schr\"{o}dinger equation on $M$.
\begin{thm}
\label{t:observe}
Let $\Omega\subset M$ be any non-empty open set and $T>0$, then there exists a constant $K>0$ depending only on $M$, $\Omega$ and $T$, such that for any $u_0\in L^2(M)$, we have
\begin{equation}
\label{e:observe}
\|u_0\|_{L^2(M)}^2\leq K\int_0^T\|e^{it\Delta}u_0\|_{L^2(\Omega)}^2dt.
\end{equation}
\end{thm} 

The following control result for the Schr\"{o}dinger equation then follows immediately by the HUM method of Lions \cite{lions}.
\begin{thm}
\label{t:control}
Let $\Omega\subset M$ be any non-empty open set and $T>0$, then for any $u_0\in L^2(M)$, there exists $f\in L^2((0,T)\times\Omega)$ such that the solution of the equation
\begin{equation}
(i\partial_t+\Delta)u(t,x)=f1_{(0,T)\times\Omega}(t,x),
\quad 
u(0,x)=u_0(x)
\end{equation}
satisfies
\begin{equation}
u(T,x)\equiv0.
\end{equation}
\end{thm}

\begin{rem}
In fact, by an elementary perturbation argument, it is not hard to see that Theorem \ref{t:semi-control} still holds if we replace the Laplacian operator $-\Delta$ by a general Schr\"{o}dinger operator $-\Delta+V$ with $V\in L^\infty(M;\mathbb{R})$. Following the proof, we can also replace $-\Delta$ by $-\Delta+V$ in Theorem \ref{t:observe} and \ref{t:control}. It is interesting to ask if this can be further extended to $L^2$-potentials as in the case of two-dimensional tori \cite{bbz}. Another interesting question is to extend the result to rough control sets as in \cite{bz17}.
\end{rem}

\subsection{Control for Schr\"{o}dinger equations}

In general, the pioneering work of Lebeau \cite{lebeau} showed that control for Schr\"{o}dinger equation holds under the geometric control condition (see \cite{blr}):
\begin{equation}
\label{e:geocontrol}
\text{There exists } L=L(M,\Omega)>0 \text{ s.t. every geodesic of length } L \text{ on } M \text{ intersects } \Omega.
\end{equation}
This geometric control condition is necessary when the geodesic flow is periodic (e.g. $M$ is a sphere), see Macia \cite{macia}. However in general, it is not necessary for observability and control for Schr\"{o}dinger equation. In fact, Theorem \ref{t:observe} and Theorem \ref{t:control} show that no condition is needed for the nonempty open set $\Omega$ on a compact hyperbolic surface. 

To our best knowledge, the only other manifold on which this is true is the flat torus. This is first proved by Jaffard \cite{jaffard} and Haraux \cite{haraux} in dimension two and by Komornik \cite{komornik} in higher dimensions. These results are further extended to Schr\"{o}dinger operators $-\Delta+V$ with smooth potential $V$ by Burq--Zworski \cite{bz12}, and $L^2$-potential $V$ by Bourgain--Burq--Zworski \cite{bbz} in dimension two; some class of potentials $V$ including continuous ones by Anantharaman--Macia \cite{am} in any dimension. We also mention the recent result of Burq--Zworski \cite{bz17} on control on two-dimensional tori by any $L^4$ functions or sets with positive measures. 

For certain other manifolds, observability and control for Schr\"{o}dinger equation is known under weaker dynamical conditions. For example, Anantharaman-Rivi\`{e}re \cite{ar} proved the case where $M$ is a manifold with negative sectional curvature and $\Omega$ satisfies an entropy condition, i.e. the set of uncontrolled trajectories is thin. A similar dynamical condition appears in the work of Schenck \cite{schenck} on energy decay of wave equation on such manifolds. In the case of manifolds with boundary, Anantharaman--L\'{e}autaud--Macia \cite{alm} showed the control and observability for Schr\"{o}dinger equation on the disk by any nonempty open set touching the boundary. 

\subsection{Control for wave equation}
For control of wave equation, by propagation of singularities, the geometric control condition \eqref{e:geocontrol} is necessary and sufficient, see Bardos--Lebeau--Rauch \cite{blr} and Burq--Gerard \cite{bg}.

We remark that the same argument as in Proposition \ref{p:semi-observe} (or the abstract result in Burq--Zworski \cite{bz04}) gives the following semiclassical observability result from Theorem \ref{t:semi-control}.
\begin{prop}
\label{p:semi-observe-wave}
Let $\chi\in C_0^\infty((\frac{1}{2},2))$, then there exists $C$, $K$ and $h_0>0$ such that for all $0<h<h_0$, $u_0\in L^2(M)$, we have
\begin{equation}
\label{e:semi-observe-wave}
\|\chi(h\sqrt{-\Delta})u_0\|_{L^2(M)}^2\leq \frac{K}{C\log(1/h)}\int_0^{C\log(1/h)}\|e^{it\sqrt{-\Delta}}\chi(h\sqrt{-\Delta})u_0\|_{L^2(\Omega)}^2dt.
\end{equation}
\end{prop}
However it is unclear to us at the moment whether the HUM method gives a control result for some explicit subspace of $L^2$-functions.

\subsection{Notations}
We recall some notations from semiclassical analysis and refer to the book \cite{semi} for further references. First, the semiclassical Fourier transform on $\mathbb{R}$ is defined by
\begin{equation*}
\mathcal{F}_h\varphi(\tau)=\int_{\mathbb{R}}e^{-it\tau/h}\varphi(t)dt
\end{equation*}
and its adjoint is given by
\begin{equation}
\label{e:semi-fourier}
\mathcal{F}_h^\ast\varphi(\tau)=\int_{\mathbb{R}}e^{it\tau/h}\varphi(t)dt.
\end{equation}
The Parseval identity show that
\begin{equation}
\label{e:parseval}
\|\mathcal{F}_h\varphi\|_{L^2}=\|\mathcal{F}_h^\ast\varphi\|_{L^2}=(2\pi h)^{1/2}\|\varphi\|_{L^2}.
\end{equation}
We also use the standard quantization $a(t,D_t)$ on $\mathbb{R}$ and fix a semiclassical quantization $\Op_h(a)$ on $M$. We refer to \cite{semi} for the standard definition and properties. Finally, as usual, $C$ denotes a constant which may change from line to line.

\subsection*{Acknowledgement}
I would like to thank the referee for many helpful suggestions. I am very grateful to Nicolas Burq, Semyon Dyatlov and Maciej Zworski for the encouragement and discussions on the topic. I would also like to thank BICMR at Peking University, Sun Yat-Sen University and YMSC at Tsinghua University for the hospitality during the visit where parts of the notes were finished.

\section{Proof of the theorems}

All the parts of the proof are well known in the literature. Here we present them in a self-contained way.

\subsection{Semiclassical observability}
We first prove a semiclassical version of the observability result
\begin{prop}
\label{p:semi-observe}
Let $\chi\in C_0^\infty((\frac{1}{2},2))$, $\psi\in C_0^\infty(\mathbb{R};[0,1])$ not identically zero, then there exist $C$, $h_0>0$ such that for all $0<h<h_0$, $u_0\in L^2(M)$, we have
\begin{equation}
\label{e:semi-observe-g}
\|\chi(-h^2\Delta)u_0\|_{L^2(M)}^2\leq C\int_{\mathbb{R}}\|\psi(t)e^{it\Delta}\chi(-h^2\Delta)u_0\|_{L^2(\Omega)}^2dt.
\end{equation}
In particular, for any $T>0$, there exists $C$, $h_0>0$ such that for all $0<h<h_0$, $u_0\in L^2(M)$, we have
\begin{equation}
\label{e:semi-observe}
\|\chi(-h^2\Delta)u_0\|_{L^2(M)}^2\leq C\int_0^T\|e^{it\Delta}\chi(-h^2\Delta)u_0\|_{L^2(\Omega)}^2dt.
\end{equation}
\end{prop}
\begin{proof}
This follows directly from the abstract result in Burq--Zworski \cite[Theorem 4]{bz04} with $G(h)=C\log(1/h)$, $g(h)=C$ and $T(h)=1/h$. We present the argument in this concrete situation.

First, we put $v(t)=e^{ith\Delta}\chi(-h^2\Delta)u_0$ and write $w(t)=\psi(ht)v(t)$. It is clear that $v(t)$ solves the semiclassical Schr\"{o}dinger equation 
$(ih\partial_t+h^2\Delta)v=0$ and thus
$$(ih\partial_t+h^2\Delta)w=ih^2\psi'(ht)v(t).$$
We take the (adjoint) semiclassical Fourier transform \eqref{e:semi-fourier} to get
\begin{equation}
\label{e:fouruer-w}
(-h^2\Delta-\tau)\mathcal{F}_h^{\ast}w(\tau)
=-ih^2\mathcal{F}_h^{\ast}(\psi'(ht)v(t))(\tau).
\end{equation}
For $\tau\in(\frac{1}{2},2)$, we use \eqref{e:semi-control} and choose $a\in C_0^\infty$ to be supported in $\{(x,\xi):x\in\Omega\}$ with $\|a\|_{L^\infty}\leq1$. We then choose $\chi\in C_0^\infty(\Omega;[0,1])$ and regard it also as a function on $T^\ast M$, such that $\chi\equiv1$ on a neighborhood of $\supp a$, then for any $u\in L^2(M)$,
$$\Op_h(a)u=\Op_h(a)(\chi u)+\Op_h(a)(1-\chi)u,$$
where $\Op_h(a)(1-\chi)=\mathcal{O}_{L^2\to L^2}(h^\infty)$, so
$$\|\Op_h(a)u\|_{L^2(M)}\leq C\|\chi u\|_{L^2(M)}+\mathcal{O}(h^\infty)\|u\|_{L^2(M)}.$$
Now \eqref{e:semi-control} gives that, for $0<h<h_0$,
$$\|u\|_{L^2(M)}\leq C\|u\|_{L^2(\Omega)}+C\frac{\log(1/h)}{h}\|(-h^2\Delta-1)u\|_{L^2(M)}.$$
We can further rescale this estimate to show that uniformly for $\tau\in[1/2,2]$,
\begin{equation}
\label{e:semi-control-scale}
\|u\|_{L^2(M)}\leq C\|u\|_{L^2(\Omega)}+C\frac{\log(1/h)}{h}\|(-h^2\Delta-\tau)u\|_{L^2(M)}.
\end{equation}

For $\tau\in[\frac{1}{2},2]$, applying \eqref{e:semi-control-scale} to $u=\mathcal{F}_h^{\ast}w(\tau)$, we obtain
\begin{equation}
\label{e:w-fourier-1}
\|\mathcal{F}_h^{\ast}w(\tau)\|_{L^2(M)}
\leq C\|\mathcal{F}_h^{\ast}w(\tau)\|_{L^2(\Omega)}
+Ch\log(1/h)\|\mathcal{F}_h^{\ast}(\psi'(ht)v(t))(\tau)\|_{L^2(M)}.
\end{equation}

For $\tau\not\in[\frac{1}{2},2]$, by definition,
$$\mathcal{F}_h^{\ast}w(\tau)=\int_{\mathbb{R}}e^{-it(-h^2\Delta-\tau)/h}\psi(ht)\chi(-h^2\Delta)u_0dt.$$
Writing 
$$e^{-it(-h^2\Delta-\tau)/h}=(h^2\Delta+\tau)^{-N}(hD_t)^Ne^{-it(-h^2\Delta-\tau)/h},$$
and noting that for any $u_0\in L^2(M)$, by functional calculus,
$$\|(h^2\Delta+\tau)^{-N}\chi(-h^2\Delta)u_0\|_{L^2(M)}\leq C_N\langle\tau\rangle^{-N}\|\chi(-h^2\Delta)u_0\|_{L^2(M)},$$
we can integrate by parts repeatedly to get
\begin{equation}
\label{e:w-fourier-2}
\|\mathcal{F}_h^{\ast}w(\tau)\|_{L^2(M)}
=\mathcal{O}((h\langle\tau\rangle^{-1})^\infty)\|\chi(-h^2\Delta)u_0\|_{L^2(M)}.
\end{equation}

Combining \eqref{e:w-fourier-1} and \eqref{e:w-fourier-2}, we have the following estimate
\begin{equation*}
\begin{split}
\|\mathcal{F}_h^{\ast}w(\tau)\|_{L^2(\mathbb{R}_\tau,L^2(M))}^2
\leq&\; C\|\mathcal{F}_h^{\ast}w(\tau)\|_{L^2(\mathbb{R}_\tau, L^2(\Omega))}^2\\
&\; +C(h\log(1/h))^2\|\mathcal{F}_h^{\ast}(\psi'(ht)v(t))(\tau)\|_{L^2(\mathbb{R}_\tau, L^2(M))}^2\\
&\;+\mathcal{O}(h^\infty)\|\chi(-h^2\Delta)u_0\|_{L^2(M)}^2.
\end{split}
\end{equation*}
By the Parseval identity \eqref{e:parseval}, we have
\begin{equation}
\label{e:w-control}
\begin{split}
\|w\|_{L^2(\mathbb{R}_t,L^2(M))}^2
\leq &\;C\|w\|_{L^2(\mathbb{R}_t,L^2(\Omega))}^2
+C(h\log(1/h))^2\|\psi'(ht)v(t)\|_{L^2(\mathbb{R}_t,L^2(M))}^2\\
&\;+\mathcal{O}(h^\infty)\|\chi(-h^2\Delta)u_0\|_{L^2(M)}^2.
\end{split}
\end{equation}

From the definition of $v$ and $w$, we see
\begin{equation*}
\begin{split}
\|w\|_{L^2(\mathbb{R}_t,L^2(M))}^2=&\;\int_\mathbb{R}\psi(ht)^2\|e^{ith\Delta}\chi(-h^2\Delta)u_0\|_{L^2(M)}^2dt\\
=&\;\left(\int_\mathbb{R}\psi(ht)^2dt\right)\|\chi(-h^2\Delta)u_0\|_{L^2(M)}^2\\
=&\;h^{-1}\|\psi\|_{L^2(\mathbb{R})}^2\|\chi(-h^2\Delta)u_0\|_{L^2(M)}^2.
\end{split}
\end{equation*}
\begin{equation*}
\begin{split}
\|w\|_{L^2(\mathbb{R}_t,L^2(\Omega))}^2=&\;\int_\mathbb{R}\psi(ht)^2\|e^{ith\Delta}\chi(-h^2\Delta)u_0\|_{L^2(\Omega)}^2dt\\
=&\;h^{-1}\int_{\mathbb{R}}\|\psi(t)e^{it\Delta}\chi(-h^2\Delta)u_0\|_{L^2(\Omega)}^2dt
\end{split}
\end{equation*}
and
\begin{equation*}
\begin{split}
\|\psi'(ht)v(t)\|_{L^2(\mathbb{R}_t,L^2(M))}^2
=&\;\int_\mathbb{R}|\psi'(ht)|^2\|e^{ith\Delta}\chi(-h^2\Delta)u_0\|_{L^2(M)}^2dt\\
=&\;\left(\int_\mathbb{R}|\psi'(ht)|^2dt\right)\|\chi(-h^2\Delta)u_0\|_{L^2(M)}^2\\
=&\;h^{-1}\|\psi'\|_{L^2(\mathbb{R})}^2\|\chi(-h^2\Delta)u_0\|_{L^2(M)}^2.
\end{split}
\end{equation*}
As long as $h$ is small and $\psi\not\equiv0$, we can absorb the last two terms on the right-hand side of \eqref{e:w-control} into the left-hand side and conclude the proof.
\end{proof}

\subsection{Observability with error}
Now we prove Theorem \ref{t:observe} with an error in $H^{-4}(M)$.
\begin{prop}
\label{p:observe-error}
There exists a constant $C>0$ such that for any $u_0\in L^2(M)$, we have
\begin{equation}
\label{e:observe-error}
\|u_0\|_{L^2(M)}^2\leq C\left(\int_0^T\|e^{it\Delta}u_0\|_{L^2(\Omega)}^2dt+\|u_0\|_{H^{-4}(M)}^2\right).
\end{equation}
\end{prop}
\begin{proof}
Again, this argument can be found in Burq--Zworski \cite[Theorem 7]{bz04} or \cite[Proposition 4.1]{bz12}. To pass from the semiclassical observability to the classical one, we use a dyadic decomposition
\begin{equation*}
1=\varphi_0(r)^2+\sum_{k=1}^\infty\varphi_k(r)^2
\end{equation*}
where
\begin{equation*}
\varphi_0\in C_0^\infty((-2,2);[0,1]),\quad
\varphi_k(r)=\varphi(2^{-k}|r|),\quad
\varphi\in C_0^\infty((1/2,2);[0,1]).
\end{equation*}
Then we have
\begin{equation}
\label{e:l2norm}
\|u_0\|_{L^2(M)}^2=\sum_{k=0}^\infty\|\varphi_k(-\Delta)u_0\|_{L^2(M)}^2,
\end{equation}
and
\begin{equation}
\label{e:h2norm}
\|u_0\|_{H^{-4}(M)}^2=\|(-\Delta+1)^{-2}u_0\|_{L^2(M)}^2\sim\sum_{k=0}^\infty2^{-4k}\|\varphi_k(-\Delta)u_0\|_{L^2(M)}^2.
\end{equation}

Fix an integer $K$ so that $2^{-K}<h_0^2$,  then for $k\geq K$, by \eqref{e:semi-observe-g}, we have
\begin{equation}
\label{e:largek}
\|\varphi_k(-\Delta)u_0\|_{L^2(M)}^2\leq C\int_{\mathbb{R}}\|\psi(t)e^{it\Delta}\varphi_k(-\Delta)u_0\|_{L^2(\Omega)}^2dt
\end{equation}
uniformly in $k$ where we choose $\psi\in C_0^\infty((0,T);[0,1])$.

The idea is to use the Schr\"{o}dinger equation to change the frequency localization in space $\varphi_k(-\Delta)$ to frequency localization in time $\varphi_k(D_t)$. More precisely, since $(D_t-\Delta)e^{it\Delta}=0$ and all $\varphi_k$ are even, we have
$$e^{it\Delta}\varphi_k(-\Delta)u_0
=\varphi_k(-\Delta)e^{it\Delta}u_0=
\varphi_k(-D_t)e^{it\Delta}u_0=\varphi_k(D_t)e^{it\Delta}u_0.$$
Now we introduce another cutoff function in time $\widetilde{\psi}\in C_0^\infty((0,T);[0,1])$ such that $\widetilde{\psi}=1$ on a neighborhood of $\supp\psi$. This allows us to express the pseudolocality of $\psi(t)\varphi_k(D_t)$ as follows:
$$\psi(t)\varphi_k(D_t)=\psi(t)\varphi_k(D_t)\widetilde{\psi}(t)+E_k(t,D_t)$$
where $E_k(t,D_t)=\psi(t)[\tilde{\psi}(t),\varphi(2^{-k}D_t)]$ with symbol satisfying
\begin{equation}
\label{e:eksymbol}
\partial^\alpha E_k(t,\tau)=\mathcal{O}(2^{-kN}\langle t\rangle^{-N}\langle\tau\rangle^{-N}),\quad \forall N.
\end{equation}

Now we have
\begin{equation*}
\begin{split}
&\;\|\psi(t)e^{it\Delta}\varphi_k(-\Delta)u_0\|_{L^2(\Omega)}^2
=\|\psi(t)\varphi_k(D_t)e^{it\Delta}u_0\|_{L^2(\Omega)}^2\\
\leq&\; \|\psi(t)\varphi_k(D_t)\widetilde{\psi}(t)e^{it\Delta}u_0\|_{L^2(\Omega)}^2+\|E_k(t,D_t)e^{it\Delta}u_0\|_{L^2(\Omega)}^2
\end{split}
\end{equation*}
Therefore by \eqref{e:l2norm} and \eqref{e:largek}, we get
\begin{equation*}
\begin{split}
\|u_0\|_{L^2(M)}^2
\leq&\; \sum_{k=0}^{K-1}\|\varphi_k(-\Delta)u_0\|_{L^2(M)}^2
+\sum_{k=K}^\infty C\int_{\mathbb{R}}\|\varphi_k(D_t)\widetilde{\psi}(t)e^{it\Delta}u_0\|_{L^2(\Omega)}^2dt\\
&\;\;\;\;\;+\sum_{k=K}^\infty C\int_{\mathbb{R}}\|E_k(t,D_t)e^{it\Delta}u_0\|_{L^2(\Omega)}^2dt.
\end{split}
\end{equation*}
By \eqref{e:h2norm}, we see that the first sum is bounded by $C\|u_0\|_{H^{-4}(M)}^2$. The second sum is bounded by
\begin{equation*}
C\int_{\mathbb{R}}\sum_{k=0}^\infty\langle\varphi_k(D_t)^2\widetilde{\psi}(t)e^{it\Delta}u_0,\widetilde{\psi}(t)e^{it\Delta}u_0\rangle_{L^2(\Omega)}dt
=C\int_{\mathbb{R}}\|\widetilde{\psi}(t)e^{it\Delta}u_0\|_{L^2(\Omega)}^2dt
\end{equation*}
The final sum is bounded by
\begin{equation}
\label{e:finalsum}
C\sum_{k=K}^\infty\int_{\mathbb{R}}\|E_k(t,D_t)e^{it\Delta}u_0\|_{L^2(M)}^2dt
=C\sum_{k=K}^\infty\|E_k(t,D_t)e^{it\Delta}u_0\|_{L^2(\mathbb{R}\times M)}^2
\end{equation}
To show this is also bounded by $C\|u_0\|_{H^{-4}(M)}^2$, we write
\begin{equation*}
E_k(t,D_t)e^{it\Delta}u_0=E_k(t,D_t)(-D_t+1)^2e^{it\Delta}(-\Delta+1)^{-2}u_0
=\widetilde{E}_k(t,D_t)\langle t\rangle^{-2}e^{it\Delta}(-\Delta+1)^{-2}u_0
\end{equation*}
where the symbol of $\widetilde{E}_k(t,D_t)=E_k(t,D_t)(-D_t+1)^2\langle t\rangle^2$ also satisfies \eqref{e:eksymbol} and thus $\widetilde{E}_k(t,D_t)=\mathcal{O}(2^{-k}):L^2(\mathbb{R})\to L^2(\mathbb{R})$. Therefore \eqref{e:finalsum} is bounded by
\begin{equation*}
C\sum_{k=K}^{\infty} 2^{-2k}\|\langle t\rangle^{-2}e^{it\Delta}(-\Delta+1)^{-2}u_0\|_{L^2(\mathbb{R}\times M)}^2\leq C\|(-\Delta+1)^{-2}u_0\|_{L^2(M)}^2=C\|u_0\|_{H^{-4}(M)}^2.
\end{equation*}
This finishes the proof of \eqref{e:observe-error}.
\end{proof}

\begin{rem}
From the proof, it is clear that the $H^{-4}$ error can be replaced by any $H^{-m}$ error as long as $m>0$. Here we take $m=4$ to make the uniqueness-compactness argument in the next section simpler. 
\end{rem}

\subsection{Removing the error}
To finish the proof, we use the classical uniqueness-compactness argument of Bardos--Lebeau--Rauch \cite{blr} to remove the $H^{-4}$ error term. We remark that a quantitative version of the uniqueness-compactness argument is presented in \cite[Appendix A]{bbz} which can be used to remove any $H^{-m}$ error and compute the constant $K$ in \eqref{e:observe} from the constant $C$ in \eqref{e:observe-error} in principle.

For any $T>0$, consider the following closed subspaces of $L^2(M)$:
\begin{equation*}
N_T=\{u_0\in L^2(M): e^{it\Delta}u_0\equiv0 \text{ on } (0,T)\times\Omega\}.
\end{equation*}

\begin{lem}
We have $N_T=\{0\}$.
\end{lem}
\begin{proof}
In fact, if $u_0\in N_T$, then
\begin{equation*}
v_{\varepsilon,0}:=\frac{1}{\varepsilon}(e^{i\epsilon\Delta}-I)u_0\in N_{T-\delta}
\end{equation*}
if $\varepsilon\leq\delta$. Moreover, $v_{\varepsilon,0}$ converges to $v_0=i\Delta u_0$ in $L^2(M)$. To see this, we only need to show that $v_{\varepsilon,0}$ is a Cauchy sequence in $L^2(M)$. We write the orthonormal expansion of $u_0$ in terms of the Laplacian eigenfunctions 
\begin{equation*}
u_0=\sum_{j=0}^\infty u_{0,j}e_j,
\end{equation*}
where $\{e_j\}_{j=0}^\infty$ is an orthonormal basis of $L^2(M)$ formed by Laplacian eigenfunctions:
\begin{equation*}
-\Delta e_j=\lambda_je_j,\quad \|e_j\|_{L^2(M)}=1, \quad
0=\lambda_0<\lambda_1\leq\lambda_2\leq\cdots\leq\lambda_j\leq\cdots,
\quad \lambda_j\nearrow\infty.
\end{equation*}
Then for $\alpha,\beta\in(0,T/2)$, we have by \eqref{e:observe-error} (with $T$ replaced by $T/2$),
\begin{equation*}
\begin{split}
\|v_{\alpha,0}-v_{\beta,0}\|_{L^2}^2
\leq&\; C\|v_{\alpha,0}-v_{\beta,0}\|_{H^{-4}}^2\\
\leq&\; C\sum_{j=1}^\infty \left|\frac{e^{-i\alpha\lambda_j}-1}{\alpha}-\frac{e^{-i\beta\lambda_j}-1}{\beta}\right|^2(1+\lambda_j)^{-4}|u_{0,j}|^2\\
\leq&\; C\sum_{j=1}^\infty|\alpha-\beta|^2\lambda_j^4(1+\lambda_j)^{-4}|u_{0,j}|^2\leq C|\alpha-\beta|^2\|u\|_{L^2(M)}^2.
\end{split}
\end{equation*}

Now $v_0=i\Delta u_0\in N_{T-\delta}$ for any $\delta>0$, thus also in $N_T$. As a consequence, $N_T$ is an invariant subspace of $\Delta$ in $L^2(M)$. Also, by Proposition \ref{p:observe-error}, the $H^{-4}(M)$-norm is equivalent to the $L^2(M)$-norm on $N_T$, so the unit ball in $N_T$ is compact and thus $N_T$ is of finite dimension. If it is not $\{0\}$, then it must contain some Laplacian eigenfunction $\varphi$. But this would mean that $\varphi\equiv0$ on $\Omega$, which violates the unique continuation for Laplacian eigenfunctions. Therefore $N_T=\{0\}$.
\end{proof} 

Now we can proceed by contradiction to finish the proof of Theorem \ref{t:observe}. Suppose \eqref{e:observe} is not true, then we can find a sequence $\{u_{n,0}\}$ in $L^2(M)$ such that
\begin{equation}
\|u_{n,0}\|_{L^2(M)}=1,\quad \text{and} \quad 
\int_0^T\|e^{it\Delta}u_{n,0}\|_{L^2(\Omega)}^2dt\leq n^{-1}.
\end{equation}
Then we can extract a subsequence $u_{n_k,0}$ converging to $u_0$ weakly in $L^2(M)$, thus strongly in $H^{-4}(M)$. On one hand, by Proposition \ref{p:observe-error} again, we see
\begin{equation*}
1=\|u_{n_k,0}\|_{L^2(M)}^2
\leq C\int_0^T\|e^{it\Delta}u_{n_k,0}\|_{L^2(\Omega)}^2dt+C\|u_{n_k,0}\|_{H^{-4}(M)}^2\leq Cn_k^{-1}+C\|u_{n_k,0}\|_{H^{-4}(M)}^2
\end{equation*}
and thus let $k\to\infty$, we get $\|u_0\|_{H^{-4}(M)}\geq C^{-1/2}>0$. On the other hand, $u_0$ must lie in $N_T$ and thus $u_0\equiv0$. This contradiction finishes the proof of Theorem \ref{t:observe}.

\subsection{From observability to control: Hilbert Uniqueness Method (HUM)}
Now we recall how the Hilbert Uniqueness Method of Lions \cite{lions} shows that Theorem \ref{t:observe} implies Theorem \ref{t:control}.

Consider the following operators $R:L^2([0,T]\times\Omega)\to L^2(M)$ and $S:L^2(M)\to L^2([0,T]\times\Omega))$ defined by $Rg=u|_{t=0}$ where $u$ is the solution to
\begin{equation}
\label{e:defr}
(i\partial_t+\Delta)u=g1_{[0,T]\times\Omega},
\quad u|_{t=T}\equiv0
\end{equation}
and $Su_0=e^{it\Delta}u_0|_{[0,T]\times\Omega}$.

\begin{prop}
$R$ and $S$ are continuous and $R^\ast=-iS$, i.e. for any $g\in L^2((0,T)\times\Omega)$ and $u_0\in L^2(M)$,
\begin{equation}
\label{e:duality}
\langle Rg,u_0\rangle_{L^2(M)}=i\langle g,Su_0\rangle_{L^2((0,T)\times\Omega)}.
\end{equation}
In particular, the following statements are equivalent:\\
(a) (Control) $R$ is surjective;\\
(b) (Observability) There exists $c>0$ such that for all $u_0\in L^2(M)$,
\begin{equation}
\|Su_0\|_{L^2((0,T)\times\Omega)}\geq c\|u_0\|_{L^2(M)}.
\end{equation}
\end{prop}
\begin{proof}
Let $u$ be the solution to \eqref{e:defr} and $v=e^{it\Delta}u_0$, then integration by parts gives
$$\langle g,Su_0\rangle_{L^2((0,T)\times\Omega)}=\int_{[0,T]\times M} (i\partial_t+\Delta)u\cdot\bar{v}dtdx
=i\int_M \left.u\bar{v}\right|_{t=0}^{t=T}dx+\int_{[0,T]\times M}u\cdot(-i\partial_t+\Delta)\bar{v}dtdx$$
By definition of $R$, we see that
$$i\int_M \left.u\bar{v}\right|_{t=0}^{t=T}dx=-i\langle Rg,u_0\rangle_{L^2(M)}$$
while $(-i\partial_t+\Delta)\bar{v}=0$. This finishes the proof of \eqref{e:duality}. The equivalence of (a) and (b) follows by standard functional analysis argument.
\end{proof}

% arXiv bibliography macro
\def\arXiv#1{\href{http://arxiv.org/abs/#1}{arXiv:#1}}

\end{document}